\documentclass[a4paper,12pt]{amsart}

\usepackage{amsmath}
\usepackage{amssymb}
\usepackage{enumitem}
\usepackage{amsfonts}
\usepackage{graphicx}
\usepackage{mathtools}
\usepackage[colorlinks]{hyperref}
\renewcommand\eqref[1]{(\ref{#1})} 

\usepackage{mathrsfs}
\graphicspath{ {images/} }
\def\G{{\mathbb{G}}}
\def\HN{{\mathbb{H}^{n}}}

\def\L{{\mathcal{L}}}
\def\R{{\mathbb{R}}}
\def\pl{{\pi_{\lambda}}}

\def\h{{\dot{\mathcal{H}}^{2}_{\mathcal{L}}(\G)}}

\DeclareMathOperator{\Tr}{Tr}
\title[solvability of non-invariant equations  on stratified groups]{Nearness and solvability of non-invariant equations on stratified groups}

 \author[M. Chatzakou]{Marianna Chatzakou}
\address{
	Marianna Chatzakou:
	\endgraf
	Department of Mathematics: Analysis, Logic and Discrete Mathematics
	\endgraf
	Ghent University, Krijgslaan 281, Building S8, B 9000 Ghent
	\endgraf
	Belgium
	\endgraf
	{\it E-mail address} {\rm marianna.chatzakou@ugent.be}}

\author[M. Ruzhansky]{Michael Ruzhansky}
\address{
  Michael Ruzhansky:
  \endgraf
  Department of Mathematics: Analysis, Logic and Discrete Mathematics
  \endgraf
  Ghent University, Belgium
  \endgraf
 and
  \endgraf
  School of Mathematical Sciences
  \endgraf
  Queen Mary University of London
  \endgraf
  United Kingdom
  \endgraf
  {\it E-mail address} {\rm michael.ruzhansky@ugent.be}
  }
  \author[N. Yannakakis]{Nikos Yannakakis}
\address{
	Nikos Yannakakis
	\endgraf
	Department of Mathematics
	\endgraf
	National Technical University of Athens
 \endgraf
 Iroon Polytexneiou 9, 15780 Zografou, Greece
	\endgraf
	{\it E-mail address} {\rm  nyian@math.ntua.gr}}

\newtheoremstyle{theorem}
{10pt}          
{10pt}  
{\sl}  
{\parindent}     
{\bf}  
{. }    
{ }    
{}     
\theoremstyle{theorem}

\numberwithin{equation}{section}
\theoremstyle{plain}
\newtheorem{thm}{Theorem}[section]
\newtheorem{prop}[thm]{Proposition}
\newtheorem{cor}[thm]{Corollary}

\theoremstyle{definition}
\newtheorem{defn}[thm]{Definition}
\newtheorem{rem}[thm]{Remark}

\newtheoremstyle{defi}
{10pt}          
{10pt}  
{\rm}  
{\parindent}     
{\bf}  
{. }    
{ }    
{}     
\theoremstyle{defi}

\begin{document}
 	\begin{abstract}
We prove the well-posedness of the differential equation $Au=f$ in the setting of a stratified group $\G$ when the considered second-order differential operator $A$ can be non-invariant and non-linear. Our approach follows the Campanato theory of nearness of operators, allowing one to treat equations with only bounded coefficients, without any regularity assumptions. Our analysis becomes explicit in the particular case of the Heisenberg group $\HN$ of any dimension and on the Euclidean case $\R^n$, where in the latter case our results also extend the known results by treating the unbounded domain setting.
	\end{abstract}
	\maketitle

\section{Introduction}\label{SEC:intro}
The framework we consider here is that of stratified Lie groups $\G$. In this setting we consider non-invariant differential operators $A:V \rightarrow X $, where $V$ is a set of functions on $\G$, and $X$ is a real Banach space. Given such $A$, we prove the uniqueness and existence of solutions to the corresponding non-linear differential equations on $\G$; i.e., for such a differential operator $A$, we consider an equation of the form 
\[
Au=f\,,\quad f \in X\,,
\]
and prove that it is satisfied for some unique $u \in V$. The approach that we follow is to prove that $A$ is, in some appropriate sense, ``close enough'' to an operator for which we know the well-posedeness of the corresponding differential equation.

This notion of ``closeness'', formally defined as \textit{nearness of operators}, was developed to determine the well-posedeness of Dirichlet problems of non-linear elliptic equations, see e.g. \cite{GT01}, \cite{DKY20}, \cite{BD02}, \cite{K17} in the Euclidean setting, where naturally the Laplacian should play the role of the operator to which close enough is the elliptic one. In the stratified case, see \cite{DM05}, the authors approximate the $p$-Laplacian on a bounded domain in the Heisenberg group $\HN$ by some non-homogeneous, non-invariant differential operator on $\HN$. In analogy to the Euclidean setting,  the operator that we consider is the hypoelliptic sub-Laplacian on the considered stratified group $\G$; i.e., sum of squares of H\"ormander vector fields on $\G$.

The aforementioned problem is related to the following question: For $V,X$ as above, let $A: V \rightarrow X$ be a mapping. Then, what do we know about the injectivity, surjectivity, or bijectivity of $A$? In response to this question, Campanato in \cite{C94} introduced the notion of \textit{near maps}; see \cite[Definition 1]{C94}, which reads as follows: 
\begin{defn}\label{DFN:near}
    We say that $A$ is near $B$ if there exist two positive constants $\alpha,K>0$, with $K \in (0,1)$, such that for all $u,v \in V$ we have 
    \[
    \|Bu-Bv-\alpha[Au-Av]\|_{X}\leq K \|Bu-Bv\|_{X}\,.
    \]
\end{defn}
 Then, questions about the nature of $A$ can be reduced to the same questions about the nature of $B$ presuming that the latter has ``better'' properties, and in particular is invertible. Campanato proved that the  injectivity, surjectivity or bijectivity property of an operator is preserved from operators near them. Formally, we have:
 \begin{thm}[\cite{C94}, Theorem 1]\label{thm.bij}
     If the operator $B:V\rightarrow X$ is bijective, surjective or injective, and the operator $B$ is near to the operator $A$, then the operator $A$ is bijective, surjective, or injective, respectively.
 \end{thm}

For our purposes, it is more convenient to use the following  characterisation of nearness of operators.
\begin{prop}[\cite{DKY20}, Proposition 1]\label{prop.near}
    The operator $A$ is near the operator $B$ if there exist positive constants $\alpha, M$ and $\mu$ with $\mu<1+ \alpha M$ such that 
    \begin{equation}
        \label{near.1}
        \|Au-Av\| \geq M \|Bu-Bv\|\,,
    \end{equation}
    and 
    \begin{eqnarray}
        \label{near.2}
         \|Bu-Bv-\alpha (Au-Av)\| \leq \mu  \|Bu-Bv\|\,,
    \end{eqnarray}
    for all $u,v$. 
\end{prop}
\begin{rem}
    Notice that condition \eqref{near.2} is less restrictive than the one in Definition \ref{DFN:near} since $\mu$ does not have to be in the range $(0,1)$.
\end{rem}
Since, as mentioned earlier, the well-posedenss of the differential equations that we consider here is proved in terms of nearness to the analogue of the Laplacian  $\L$ in our sub-Riemmanian setting. For this, we will use the hypelliptic  counterpart of the Miranda-Talenti estimate on  $\G$, see \cite[Lemma 1.2.2]{MPS00}, in the stratified setting $\G$. The latter reads as follows:
\begin{equation}\label{MT}
\|X_iX_j u\|_{L^2(\G)} \leq C_{ij} \|\L u\|_{L^2(\G)}\,, \quad C_{ij}>0\,,    
\end{equation}
where $C_{ij}>0$ depends on the H\"ormander vertor fields $X_1,\ldots,X_N$ (and so on the choice of $\G$), and is a consequence of \cite[Theorem 4.4.16(2)]{FR16}.  In what follows we will denote by $C_0$ the quantity
\begin{equation}\label{C0}
    C_0=\left(\sum_{1\leq i,j \leq N}C_{ij}^{2} \right)^{1/2}\,,
\end{equation}
where each $C_{ij}$ corresponds to the Miranda-Talenti estimate \eqref{MT} for the vector fields $X_i,X_j$ that generate the Lie algebra of $\G$. In the case where $\G=\R^n$ and $N=n$, we have $C_{ij}=1$, see c.f. \cite[Lemma 1.2.2]{MPS00}, which in turn implies that $C_0=n$.

The main result of this paper, see Theorem \ref{theorem1} below, concerns the well-posedeness of an equation with an associated differential operator expressed in terms of a function of the form
\[
a: \G \times \R^{N^2} \rightarrow \R\,,
\]
for a stratified group $\G$. This type of problem was previously considered in \cite[Theorem 3]{DKY20},  where the authors also employ the method of nearness of operators, but in \cite{DKY20} the space variable $x$ lies in $\Omega$, where $\Omega \subset \mathbb{R}^n$ is an open, bounded and convex domain of $\mathbb{R}^n$ with a regular enough boundary. In this sense, our result is already new in the Euclidean setting.

Let us now describe the assumptions on the differential operator that we consider: For a fixed stratified group $\G$, $a$ is a  Caratheodory function; i.e., measurable with respect to $x \in \G$, for all $\xi \in \R^{N^2}$, and continuous with respect to $\xi \in \R^{N^2}$ for all $x \in \G$. Additionally, we assume that there exist positive constants $M, \alpha, \gamma$ and $\delta$ such that $\sqrt{(\gamma+\delta)(\gamma C_{0}^{2}+\delta)}<1+\alpha M$, for $C_0$ as in \eqref{C0}, and the following two conditions are satisfied
\[
|a(x,\xi+\tau)-a(x,\xi)|\geq M |\Tr \tau|\tag{C1}\,,
\]
\[
|\Tr \tau -\alpha (a(x,\xi+\tau)-a(x,\xi))|\leq \gamma \|\tau\|_{\R^{N^2}}+\delta |\Tr \tau|\,,\tag{C2}
\]
for all $\xi, \tau \in \R^{N^2}$ and for almost all $x \in \G$, where $\Tr \tau$ and $\|\tau\|_{\R^{N^2}}$ stand for the trace and for the Euclidean norm of $\tau$, respectively.

In this paper, we will extensively use the homogeneous Sobolev space $\h$ which is defined as the completion of the Schwartz space $\mathcal{S}(\G)$ for the Sobolev norm \eqref{EQ:Sob}, see e.g.  \cite[Definition 4.4.12]{FR16} for its definition and further properties, and also Section \ref{SEC:2}.

Here is the first main result of this paper:

\begin{thm}\label{theorem1}
   Let $\G$ be a stratified group with $Q\geq 3$, and let $\{X_j\}_{j=1}^{N}$ be a H\"ormander system of vector fields on $\G$. If $a: \G \times \R^{N^2} \rightarrow \R$ satisfies the conditions $(C1)$ and $(C2)$, then the equation 
    \[
    a(x, \{X_iX_j\}_{i,j=1}^{N}u)=f\,, \quad f \in L^2(\G)\,,
    \]
    has a unique solution $u \in \h$.
\end{thm}

Note that conditions (C1) and (C2) on the matrix $a$ generalise the nonlinear Cordes condition introduced by Campanato \cite{C89a}. See e.g., \cite{T04} for a general overview of stronger and weaker variations of the Cordes condition, and their relation to Campanato theory of near operators. 

The second main result of this paper, see Theorem \ref{theorem2}  below,  is to show the well-posedeness of the differential equation 
\begin{equation*} 
     Au:=   \sum_{i,j=1}^{N}c_{ij}(x) X_i X_j u=f\,, 
\end{equation*}
on a stratified group $\G$ and for some suitable $c_{ij} \in L^{\infty}(\G)$, that is, if the equation $Au=f$  has a unique solution. 

To motivate our analysis, let us recall the following problem in the Euclidean setting:

Let $\Omega \subset \mathbb{R}^n$, $n>2$, be a bounded set with $\partial \Omega$ sufficiently regular, and let $A=\{a_{ij}\}$ be a real matrix, with $a_{ij} \in L^{\infty}(\Omega).$ Under this considerations, it is well-known, see \cite{P66}, that the Dirchlet problem
\begin{equation}\label{dir}
    \begin{cases}
        u \in H^{2,2}\cap H_{0}^{1,2}(\Omega)\,,\\
    \sum_{i,j=1}^{n}a_{ij}D_iD_ju(x)=f(x)\,,\quad \text{a.e.}\,\,\text{in}\,\, \Omega\,,
    \end{cases}
\end{equation}
for $f \in L^2(\Omega)$ is well-posed under a hypothesis stronger than the uniform ellipticity of $A$. In view of this requirement, Campanato in \cite{C67} imposed the following  condition, introduced by Cordes in \cite{C61}, see also \cite{T65}, on the coefficients $a_{ij}$:
\begin{equation}
    \label{cordes}
    \frac{\left(\sum_{i=1}^{n}a_{ii}(x) \right)^2}{\sum_{i,j=1}^{n}a_{ij}^{2}(x)}\geq n-1+\varepsilon\,, \quad \text{for all}\quad x \in \Omega\,,
\end{equation}
and for some $\varepsilon \in (0,1)$. If condition \eqref{cordes} is satisfied, then we say that the coefficients $(a_{ij})$ satisfy the \textit{Cordes condition} in $\Omega$.

The next theorem shows that the analogue of  \eqref{dir} has a unique solution when considered in the entire stratified group $\G$. The interesting feature of this result is that only the boundedness of coefficients $c_{ij}$ is required, without any regularity.  

\begin{thm}\label{theorem2}
     Let $\G$ be a stratified group with $Q\geq 3$, and let $\{X_j\}_{j=1}^{N}$ be a H\"ormander system of vector fields on $\G$, $f \in L^2(\G)$ and coefficients $(c_{ij}) \in L^{\infty}(\G)$, $i,j=1,\ldots,N$,  that  satisfy 
     the condition
\begin{equation}
    \label{ell}  
    \sum_{i,j=1}^N c_{ij}(x)\xi_i\xi_j\geq c_0|\xi|^2,
\end{equation} 
for some $c_0>0$, and all $x\in\G$ and $\xi\in\mathbb R^N$, and 
     the Cordes condition 
\begin{equation}
    \label{cordes2}
    \frac{\left(\sum_{i=1}^{N}c_{ii}(x) \right)^2}{\sum_{i,j=1}^{N}c_{ij}^{2}(x)}\geq N-1+\varepsilon\,, \quad \text{for all}\quad x \in \G\,,
\end{equation}     
for some $\varepsilon\in (0,1)$ such that $\sqrt{1-\varepsilon}C_0<1$, for $C_0$ given by \eqref{C0}. Then the equation 
    \begin{equation}
        \label{1}  
         Au:=   \sum_{i,j=1}^{N}c_{ij}(x) X_i X_j u=f
    \end{equation}
    has a unique solution $u \in \h$.
\end{thm}
We will refine this result with the specific values of constants in the case of the Heisenberg group, in Section \ref{SEC:4}. However, the statement of Theorem \ref{theorem2} already on $\G=\R^n$ appears to be new, in which case we also have $N=n$:
\begin{cor}
    Let $f \in L^2(\mathbb{R}^n)$ and let coefficients $(c_{ij}) \in L^{\infty}(\R^n)$, $i,j=1,\ldots,n$, satisfy condition \eqref{ell} and the Cordes condition \eqref{cordes2} for $N=n$, and for some $\varepsilon\in (0,1)$ such that $\sqrt{(1-\varepsilon)}n<1$. Then the equation 
     \[
     \sum_{i,j=1}^{n}c_{ij}(x) \partial^2_{x_i x_j} u=f
    \]
     has a unique solution $u \in \dot{H}^2(\R^n)$.   
\end{cor}
    This result follows by Theorem \ref{theorem2} since when $\G=\R^n$ we have $C_0=n$. We also note that in the case of $\G=\R^n$ condition \eqref{ell} is the ellipticity condition, while on general stratified groups $\G$ condition \eqref{ell} is a condition on the quadratic form corresponding to coefficients $c_{ij}(x),$ $1\leq i,j\leq N$.

\medskip

The proofs of Theorems \ref{theorem1} and \ref{theorem2} can be found in Section \ref{SEC:3}. Section \ref{SEC:4} is devoted to the application of our main results in the two most well-studied cases of stratified groups: the trivial case of  $\mathbb{R}^n$ and the case of the Heisenberg group $\HN$. In the latter cases, the constant $C_0$ becomes explicit,  and the condition $\sqrt{1-\varepsilon}C_0<1$ in Theorem \ref{theorem2} can be checked directly. 
In the case of $\mathbb{R}^n$ the constant  is known by the Miranda-Talenti estimate, see e.g. \cite[Lemma 1.2.2]{MPS00}, and in the case of $\HN$ it is computed in Proposition \ref{Heis}. The results contained in this paper are new even in the trivial abelian case of $\R^n$. 

\section{Preliminaries}\label{SEC:2}
In what follows we denote by $\G\equiv \R^n$ any stratified Lie group of homogeneous dimension $Q\geq n$, having the vector fields $X_1,\ldots,X_N$ generating its Lie algebra. The analysis of stratified groups, and more generally of homogeneous Lie groups, appeared in the 80's, and the seminal work of Folland and Stein \cite{FS82} has essentially initiated it.  

Recall, that  the sub-Laplacian of a stratified group $\G\equiv \R^n$ is the hypoelliptic operator
\[
\L=\sum_{i=1}^{N}X_{i}^{2}\,,
\]
that is the sum of squares of smooth real-valued vector fields satisfying H\"ormander's condition \cite{H67}.  The operator $\mathcal{L}$ is a homogeneous differential operator of degree $2$ with respect to the dilations $\delta_\lambda: \G \rightarrow \G$, $\lambda>0$, on the group $\G$. Recall that the dilations are automorphisms of $\G$, and so also of the corresponding Lie algebra $\mathfrak{g}$, and their action on $\G, \mathfrak{g}$, is determined by the stratificaton of $\mathfrak{g}$, see e.g. \cite[Section 3.1]{FR16}. The Jacobian determinant of the map $\delta_\lambda$ for $\lambda>0$ is computed as $\lambda^Q$, where $Q$ is the homogeneous dimension of $\G$. For a stratified group $\G$ we have $Q\geq \text{dim}(\G)$, where $\text{dim}(\G)$ is the topological dimension of $\G$, and the equality holds only in the trivial case where $\G=\R^n$ is a Euclidean space. 

Similar arguments as those used in the Euclidean setting show the validity of Young's inequality for the convolution of functions in our setting, and more generally in the setting of a nilpotent Lie group; see \cite[Proposition 1.5.1]{FR16}. Moreover, in \cite[Proposition 1.5.1]{FR16} it is given that, if $p,q \in (1,\infty)$ are such that $\frac{1}{p}+\frac{1}{q}>1$, $f_1 \in L^p(\G)$ and $f_2$ satisfies the weak-$L^q(\G)$ condition, i.e., it satisfies
\[
\sup_{s>0} s^q |\{ x\,:\, |f_2(x)|>s\}| := \|f_2\|^{q}_{w-L^q(\G)}<\infty\,,
\]
then we have 
\begin{equation}
    \label{Young}
    \|f_1 \ast f_2\|_{r} \leq \|f_1\|_p \|f_2\|_{w-L^q(\G)}<\infty\,,
\end{equation}
where $r$ satisfies the relation $\frac{1}{p}+\frac{1}{q}=\frac{1}{r}+1$.

Recall that for a left-invariant operator  $A$ on a Lie group $\G$ we have 
\[
\|A\|_{\text{op}}=\sup_{\pi \in \widehat{\G}}\|\pi(A)\|_{\mathcal{L}(\mathcal{H}_\pi)}\,,
\]
where $\widehat{\G}$ denotes the unitary dual of the group $\G$, $\mathcal{H}_\pi$ denotes the representation space\footnote{In our case, where $\G$ is a nilpotent Lie group, we have $\mathcal{H}_\pi=L^2(\R^k)$, for some $k \in \mathbb{N}$.} of $\pi$, and   $\|A\|_{\text{op}}$ is the operator norm of $A$ from $L^2(\G)$ to $L^2(\G)$. 

Let us recall the following result as in \cite[Theorem 3.2.45]{FR16}.
\begin{thm}[Liouville theorem on homogeneous Lie groups]\label{liouv}
    Let $L$ be a homogeneous left-invariant operator on a homogeneous Lie group $\G$. We assume that $L$ and $L^{t}$ are hypoelliptic on $\G$. If the distribution $f \in \mathcal{S}'(\G)$ satisfies $Lf=0$, then $f$ is a polynomial.
\end{thm}

The spaces of functions/distributions that we consider here are the Banach spaces $L^p(\G)$, $p \in [1,\infty)$ and the subspace of tempered distributions $\h$ defined in a more general form in \cite[Definition 4.4.12]{FR16}. Particularly, by $\h$ we have denoted the space of tempered distributions obtained by the completion of $\mathcal{S}(\G)$ for the Sobolev norm 
    \begin{equation}\label{EQ:Sob}
    \|u\|_{\h}:=\|\L u\|_{L^2(\G)}\,, \quad u \in \mathcal{S}(\G).
    \end{equation}
By \cite[Proposition 4.4.13 (1)]{FR16} the space $\h$ is a Banach space that contains $\mathcal{S}(\G)$ as a dense subspace. 

   

\section{Main results in the general case}\label{SEC:3}
We start with the solution of the linear problem.
\begin{prop}\label{prop.uniq}
    Let $\G$ be a stratified Lie group of homogeneous dimension $Q \geq 3$, and let  $f \in L^2(\G)$. There exists a unique solution $u \in \h $ to the problem 
    \begin{equation}\label{uniq}
            \L u= f\,\quad \rm on \quad \G.
    \end{equation}
    The solution operator $f\mapsto u$ is continuous from $L^2(\G)$ to $\h $.
\end{prop}
\begin{proof}
We know that $k=d^{2-Q}$, where $d$ is a symmetric homogeneous quasi-norm on $\G$ called the $\L$-gauge, is the fundamental solution of the sub-Laplacian $\L$, see, e.g. \cite{F73}, $k\in S'(\G)$, see \cite[Theorem 3.2.40]{FR16}.

Then $u=f \ast k$ solves \eqref{uniq}, the solution operator $\L^{-1}$ is also continuous as an operator $\L^{-1}:{\dot{\mathcal{H}}^{-2}_{\mathcal{L}}(\G)}\to L^2(\G)$, by \cite[Proposition 4.4.13 (4)]{FR16}. Since the dual of ${\dot{\mathcal{H}}^{-2}_{\mathcal{L}}(\G)}$ is $\h$, see \cite[Theorem 4.4.28 (10)]{FR16}, by duality we have that $\L^{-1}$ is continuous from $L^2(\G)$ to $\h $.


Now, to prove the uniqueness of the solution to the equation \eqref{uniq}, let $u_1,u_2 \in \mathcal{S}'(\G)$ be two solutions of the problem \eqref{uniq} satisfying $u_1,u_2\in \h$. Then the difference $u:=u_1-u_2$ is a solution to the problem $\L u=0$ with 
$u\in \h$. 

Since $\L=\L^{t}$ satisfies the hypothesis of Theorem \ref{liouv}, we get that $u$ is a polynomial on $\G$. It is immediate that there is no polynomial $u$ other than the one that is identical to zero satisfying $u\in \h$, by our definition of $\h$.  Hence, we must have $u_1=u_2$, and the proof is complete. 
\end{proof}

\begin{rem} We note that for $Q\geq 5$, we have the Sobolev embedding $\h\subset L^{\frac{2Q}{Q-4}}(\G)$, see \cite[Proposition 4.4.23 (5)]{FR16}. This can be also shown directly. We observe that
\begin{eqnarray*}
    |\{ x \in \G: d(x)^{2-Q}>s\}| & = &   |\{\{x \in \G : d(x)<s^{\frac{1}{2-Q}}\}| \\
   & = & \int_{\{x \in \G : s^{\frac{1}{Q-2}}d(x)<1\}}\,dx\\
   & = & \int_{\{y \in \G : d(y)<1\}}s^{\frac{Q}{2-Q}}\,dy\\
   & = & s^{\frac{Q}{2-Q}}\cdot {\rm vol}(\{y: d(y)<1\})=Cs^{\frac{Q}{2-Q}}< \infty\,,
\end{eqnarray*}
where we have used the fact that $d$ is homogeneous of degree $1$. Hence we have proved that $k \in w-L^{\frac{Q}{Q-2}}(\G)$, since
\[
s^{\frac{Q}{Q-2}}|\{ x \in \G: d(x)^{2-Q}>s\}| <\infty\,,
\]
and by  Young's convolution inequality  \eqref{Young} we get $u=f \ast k \in L^{\frac{2Q}{Q-4}}(\G)$.
\end{rem}

\begin{proof}[Proof of Theorem \ref{theorem1}]
    We will show that the operator $A:\h \rightarrow L^{2}(\G)$ defined as 
    \[
    Au(x):= a(x, \{X_iX_j\}_{i,j=1}^{N}u (x))\,,\quad x \in \G\,,
    \]is near the sub-Laplacian $\L$. Indeed, let $u,v \in \h$ and we set 
    \begin{equation}\label{t,xi}
    \tau:=\{X_iX_j(u-v)\}_{i,j=1}^{N}\,,\quad \xi:=\{X_iX_jv\}_{i,j=1}^{N}\,.
    \end{equation}
    Since 
    \begin{equation}\label{trace,t}
    \Tr \tau=\sum_{i=j}X_iX_j (u-v)=\L(u-v)\,,
    \end{equation}
    condition (C1) gives 
    \[
    | a(x, \{X_iX_j\}_{i,j=1}^{N}u(x))-a(x, \{X_iX_j\}_{i,j=1}^{N}v(x))|^2\geq M^2 |\L(u-v)(x)|^2\,,
    \]
    for almost all $x \in \G$. Integrating the latter over $\G$ with respect to the Haar measure on $\G$ we get
    \[
        \|Au-Av\|_{L^{2}(\G)} \geq M \|\L u-\L v\|_{L^{2}(\G)}\,,
    \]
    and this shows the condition \eqref{near.1}, as in Proposition \ref{prop.near}, of the nearness of the operator $A$ to the operator $\L$. On the other hand, using condition (C2) we get the following estimate that holds true for almost every $x \in \G$:
    \begin{eqnarray*}
    |\Tr \tau - \alpha (a(x,\xi+\tau)-a(x,\xi))|^2& \leq & \left(\gamma \|\tau\|_{\mathbb{R}^{N^2}}+\delta |\Tr \tau|\right)^2\\
    & \leq & (\gamma+ \delta)\left(\gamma \|X_i X_j (u-v)(x)\|^2_{\mathbb{R}^{N^2}}+\delta |\L(u-v)(x)|^2 \right)
    \end{eqnarray*} where for the second estimate we have used the inequality 
    \[
    (\gamma \alpha +\delta \beta)^2 \leq (\gamma +\delta)(\gamma \alpha^2+\delta \beta^2)\,,\quad \alpha, \beta, \gamma, \delta\geq 0\,,
    \]
    which can be easily verified, and the expressions given in \eqref{t,xi} and \eqref{trace,t}.
    Integrating the latter over $\G$, we get 
    \begin{equation}\label{gamma,delta}
    \begin{split}
       & \int_{\G} \left|\L (u-v)(x)-\alpha(a(x, \{X_iX_j\}_{i,j=1}^{N}u(x))-a(x, \{X_iX_j\}_{i,j=1}^{N}v(x)))\right|^2dx\\&\leq (\gamma+ \delta)\left[ \gamma \int_{\G}\sum_{i,j=1}^{N}|X_i X_j (u-v)(x)|^2dx+\delta \int_{\G} |\L (u-v)(x)|^2\,dx \right] \\&
        =(\gamma+\delta) \left[\gamma I+\delta II\right]\,,
    \end{split}
\end{equation}
where we have defined 
\[
I:=\int_{\G}\sum_{i,j=1}^{N}|X_i X_j (u-v)(x)|^2dx\,,
\]
and 
\begin{equation}\label{II.L2}
II:=\int_{\G} |\L (u-v)(x)|^2\,dx=\|\L (u-v)\|^2_{L^2(\G)}\,.
\end{equation}
To estimate $I$ we make use of \eqref{MT}, and we get
\begin{eqnarray}\label{I.L2}
    I & = & \sum_{i,j=1}^{N}\|X_iX_j (u-v)\|^2_{L^2(\G)}\nonumber\\
    & \leq & C_{0}^{2} \|\L(u-v)\|^2_{L^2(\G)}\,,
\end{eqnarray}
where we have set $C_{0}^2=\sum_{1 \leq i,j \leq N}C_{ij}^2$. A combination of \eqref{gamma,delta}, \eqref{II.L2} and \eqref{I.L2} gives 
\begin{equation}
    \label{gamma,delta,I,II}
    \|\L u-\L v -\alpha (Au-Av)\|_{L^2(\G)}\leq \sqrt{(\gamma+\delta)(\gamma C_{0}^{2}+\delta)}\|\L(u-v)\|_{L^2(\G)}\,.
\end{equation}
Setting $\mu=\sqrt{(\gamma+\delta)(\gamma C_{0}^{2}+\delta)}<1+\alpha M$, the latter can be rewritten as 
\[
 \|\L u-\L v -\alpha (Au-Av)\|_{L^2(\G)}\leq \mu\|\L(u-v)\|_{L^2(\G)}\,,
\]
and we have shown condition \eqref{near.2}. Summarising, both condition as in Proposition \ref{prop.near} are satisfied, implying that the operators $A$ and $\L$ are near. Since the operator $\L:\h \rightarrow L^2(\G)$ is bijective by Proposition \ref{prop.uniq}, the result follows from Theorem \ref{thm.bij} which gives that $A$ is also bijective, and the proof is complete.
\end{proof}

\begin{proof}[Proof of Theorem \ref{theorem2}]
Let $c(x):=\frac{\sum_{i=1}^{N}c_{ii}(x)}{\sum_{i,j=1}^{N}c_{ij}(x)^2}$. By the assumption \eqref{ell}, it is clear that $c(x)>0$ and $c\in L^\infty(\G)$. Given $f \in L^2(\G)$, we define the operator $T:\h\rightarrow \h$, by $Tu=v$, for $u \in \h$, with $v$ being the solution of the problem 
\begin{equation*}
        \L v=\L u-c Au+cf.
\end{equation*}
This operator is well-defined since such $v$ is unique, as a consequence of Proposition \ref{uniq}, since $\L u-c Au+cf \in L^2(\G)$, for $f$ as in the hypothesis and using the estimate \eqref{MT}. 
We will now show that the operator $T$ is a contraction as an operator from $\h$ to $\h$. Let $u_1,u_2 \in \h$. We have 
\begin{eqnarray}\label{est.T}
\|Tu_1-Tu_2\|^{2}_{\h} & = &\|v_1-v_2\|^{2}_{\h}=\|\L v_1-\L v_2\|_{L^2(\G)}^{2}\nonumber\\
& = & \| \L u_1-c Au_1+cf- (\L u_2-c Au_2+cf) \|_{L^2(\G)}^{2}\nonumber\\
& = & \|\L (u_1-u_2)-c(Au_1-Au_2) \|_{L^2(\G)}^{2}\nonumber\\
& = & \int_{\G}\left|(\L-cA)(u_1-u_2)(x) \right|^2\,dx\nonumber\\
& = & \int_{\G}\left|\sum_{i=1}^{N} X_{i}^{2}(u_1-u_2)-c(x) \sum_{i,j=1}^{N} c_{ij}(x)X_iX_j(u_1-u_2)(x)\right|^2\,dx\nonumber\\
& = & \int_{\G}\left|\sum_{i,j=1}^{N}(\delta_{ij}-c(x) c_{ij}(x))X_iX_j(u_1-u_2)(x)\right|^2\,dx\nonumber \\
& \leq & \int_{\G} \sum_{i,j=1}^{N} |\delta_{ij}-c(x)c_{ij}(x)|^2 \sum_{i,j=1}^{N} |X_i X_j (u_1-u_2)(x)|^2\,dx\nonumber\\
& = &  \|I(x)\|_{L^{\infty}(\G)} \int_{\G} II(x)\,dx\,,    
\end{eqnarray}
where we have applied the Cauchy–Schwarz inequality, and we have set 
\[
I(x):= \sum_{i,j=1}^{N} |\delta_{ij}-c(x)c_{ij}(x)|^2\,,
\]
and 
\[
II(x):=\sum_{i,j=1}^{N} |X_i X_j (u_1-u_2)(x)|^2\,.
\]
For the estimate of the norm  $\|I\|_{L^{\infty}(\G)}$, we make use of the Cordes condition \eqref{cordes} on the coefficients $(c_{ij}(x))$ and we get 
\begin{equation}
    \label{est.I}
     \|I\|_{L^{\infty}(\G)} \leq N- \frac{\left(\sum_{i=1}^{N}c_{ii}(x) \right)^2}{\sum_{i,j=1}^{N}c_{ij}^{2}(x)} \leq 1 -\varepsilon\,.
\end{equation}
Indeed, this follows from
\begin{multline*}
  \sum_{i,j=1}^{N} |\delta_{ij}-c(x)c_{ij}(x)|^2=N-2c(x)\sum_{i,j=1}^N c_{ij}(x)\delta_{ij}+c(x)^2
  \sum_{i,j=1}^N c_{ij}(x)^2 \\ =
  N-2\frac{(\sum_{i=1}^N c_{ii}(x))^2}{\sum_{i,j=1}^N c_{ij}(x)^2}+
  \frac{(\sum_{i=1}^N c_{ii}(x))^2}{\sum_{i,j=1}^N c_{ij}(x)^2}=
  N-\frac{(\sum_{i=1}^N c_{ii}(x))^2}{\sum_{i,j=1}^N c_{ij}(x)^2},
\end{multline*}
using the definition of $c(x).$

To estimate the term $ \int_{\G} II(x)\,dx$ we use the estimate \eqref{MT} and obtain
\begin{equation}
    \label{est.II}
    \int_{\G} II(x)\,dx \leq  C_{0}^{2} \|\L (u_1-u_2)\|_{L^2(\G)}^{2}=C_{0}^{2} \|u_1-u_2\|^2_{\h}\,,
\end{equation}
where $C_{0}^{2}=\sum_{1\leq i,j \leq N}C_{ij}^{2}$.
Summarising, a combination of \eqref{est.T}, \eqref{est.I} and \eqref{est.II} gives 
\[
\|Tu_1-Tu_2\|^{2}_{\h} \leq (1 -\varepsilon) C_{0}^{2} \|u_1-u_2\|^2_{\h}\,,
\]
and the latter under the condition that 
\[
\sqrt{1-\varepsilon}C_0<1\,,
\]
means that the operator $T$ is a contraction on $\h$. Hence by Banach fixed point theorem there exists unique $u \in \h$, such that $Tu=u$,  which by the definition of the operator $T$ means that for this unique $u \in \h$ we have 
\[
\L u= \L u-c Au +cf\,,
\]
 and so $u$ is the unique solution to \eqref{1},
 completing the proof. 
\end{proof}

\section{The Heisenberg group as a special case in our analysis}\label{SEC:4}
Recall that the canonical basis of the Heisenberg group $\HN\simeq \R^{2n+1}$ is given by
\[
X_i=\partial_{x_i}-\frac{y_i}{2}\partial_{t} \quad \text{and}\quad X_i=\partial_{y_i}+\frac{x_i}{2}\partial_{t}\,,\quad i=1,\ldots,n\,,
\]
\[
T=\partial_t\,.
\]
Since the only non-zero canonical commutation relation that is satisfied by the above vector fields is
\[
[X_i,Y_i]=T\,,\quad i=1,\ldots,n\,,
\]
and so the canonical positive sub-Laplacian on $\HN$ is given by 
\[
\L:=-\sum_{i=1}^n (X_{i}^{2}+Y_{i}^{2})\,.
\]
For $\lambda \in \R \setminus \{0\}$, we recall the Schr\"odinger representations $\pi_{\lambda}$ on the Heisenberg group $\HN$  
\[
\pl : \HN \rightarrow \mathcal{U}(L^2(\R^n))\,.
\]
The Schr\"odinger representations induce the infinitesimal representations on the left-invariant differential operators on $\HN$. The infinitesimal representation of such an operator $X$ is denoted by $\pl(X)$ and can be extended to operators from $L^2(\HN)$ to $L^2(\HN)$.

In what follows we denote by $\pl(X_i), \pl(Y_i)$ and $\pl(\L)$ (or by $\pi_{X_i}(\lambda)$, $\pi_{Y_i}(\lambda)$ and $\pi_{\L}(\lambda)$, respectively) both the operators and the infinite matrices associated to the corresponding vector fields with respect to the orthonormal basis consisting of the Hermite functions. For $k,l \in \mathbb{N}$ the $(k,l)$-entrices of these matrices are given in \cite{FRT21}, for $X$ and $Y$ being one of $X_i$ and $Y_j$, by 
\begin{equation*}\label{plap}
(\pl(\L))_{k,l}=|\lambda|(2k+1)\delta_{k,l},
\end{equation*}
\begin{equation}\label{px}
    (\pl(X))_{k,l}=\begin{cases}
        \sqrt{|\lambda|}\sqrt{\frac{k+1}{2}}\,, & k=l-1,\\
        -\sqrt{|\lambda|}\sqrt{\frac{k}{2}}\,, & k=l+1,\\
        0\,, & \text{otherwise},
    \end{cases}
\end{equation}
\begin{equation}\label{py}
    (\pl(Y))_{k,l}=\begin{cases}
        i\sqrt{|\lambda|}\sqrt{\frac{k+1}{2}}\,, & k=l-1,\\
        -i\sqrt{|\lambda|}\sqrt{\frac{k}{2}}\,, & k=l+1,\\
        0\,, & \text{otherwise}.
    \end{cases}
\end{equation}

The next proposition can be viewed as the Miranda-Talenti estimate for the case of the Heisenberg group $\mathbb{H}^n$. 
\begin{prop}\label{Heis}
    Let $Z_i$, $i=1,\ldots,2n$, be the generators of the Lie algebra of the Heisenberg group $\mathbb{H}^{n}$. Then we have 
    \[
    \|Z_i Z_j u\|_{L^2(\mathbb{H}^n)}\leq 2^{-1} \|\L_{\mathbb{H}^n}u\|_{L^2(\mathbb{H}^n)}\,.
    \]
\end{prop}

\begin{proof}
One can identify the following vector fields
\begin{equation*}
    Z_i=\begin{cases}
        X_i &\quad {\rm if} \quad i=1,\ldots,n\\
        Y_i&\quad {\rm if} \quad i=1=n,\ldots,2n\,.
    \end{cases}
\end{equation*}

Note that the inverse operator $\L^{-1}$ is well-defined as an operator on $L^2(\G)$, see \cite[Section 4.1.3]{FR16}. For simplicity, we will write $X$ for one of $X_i$ and $Y$ for one of $Y_j$. We have 
    \begin{eqnarray*}
        \|X Y \L^{-1}\|_{\text{op}} & = &\sup_{\lambda \in \R \setminus \{0\}} \|\pi_{XY \L^{-1}}(\lambda)\|_{\text{op}_{\lambda}}\\
        & = & \sup_{\lambda \in \R \setminus \{0\}} \max_{k,l}\left|(\pi_{XY \L^{-1}}(\lambda))_{l,k}\right|\,.
    \end{eqnarray*}
    Decomposing we get 
    \[
    \left|(\pi_{XY \L^{-1}}(\lambda))_{l,k}\right|= \sum_{l,k}(\pi_{X}(\lambda))_{l,p}(\pi_{Y}(\lambda))_{p,m}(\pi_{\L^{-1}}(\lambda))_{m,k}\,.
    \]
      Since the matrices $(\pi_{X}(\lambda))_{l,p}, (\pi_{Y}(\lambda))_{p,m}$ are upper and lower diagonal we get
    \[
    (\pi_{XY}(\lambda))_{l,m}\neq 0\quad \text{only when} \quad l \in \{m+2,m,m-2\}\,.
    \]
    And in particular we get 
     \begin{eqnarray*}
      (\pi_{XY}(\lambda))_{m,m} & =  &  \sqrt{|\lambda|}\sqrt{\frac{m+1}{2}}\left(-i \sqrt{|\lambda|}\sqrt{\frac{m+1}{2}} \right)+\left(-\sqrt{|\lambda|}\sqrt{\frac{m}{2}} \right)\left(i \sqrt{|\lambda|}\sqrt{\frac{m}{2}} \right) \\
      & = & -i|\lambda|\frac{m+1}{2}-i|\lambda|\frac{m}{2}=-i \frac{|\lambda|}{2}(2m+1)\,,
    \end{eqnarray*}
 \begin{eqnarray*}
      (\pi_{XY}(\lambda))_{m,m+2} & =  &  \sqrt{|\lambda|}\sqrt{\frac{m+1}{2}}\left(i \sqrt{|\lambda|}\sqrt{\frac{m+2}{2}} \right)\\
      & = & i \frac{|\lambda|}{2}\sqrt{(m+1)(m+2)}\,,
    \end{eqnarray*}
    and 
     \begin{eqnarray*}
      (\pi_{XY}(\lambda))_{m,m-2} & =  &  \left(-\sqrt{|\lambda|}\sqrt{\frac{m}{2}}\right)\left(-i \sqrt{|\lambda|}\sqrt{\frac{m-1}{2}} \right)\\
      & = & i \frac{|\lambda|}{2}\sqrt{m(m-1)}\,.
    \end{eqnarray*}
    Since the matrix $(\pi_{\L^{-1}}(\lambda))_{m,k}$ is diagonal, we have 
    \[
    (\pi_{XY\L^{-1}}(\lambda))_{l,k}=\sum_{\zeta} (\pi_{XY}(\lambda))_{l,\zeta} (\pi_{\L^{-1}}(\lambda))_{\zeta,k} = (\pi_{XY}(\lambda))_{l,k} (\pi_{\L^{-1}}(\lambda))_{k,k}\,.
    \]
    By the above we get 
     \[
    (\pi_{XY\L^{-1}}(\lambda))_{l,k}\neq 0\quad \text{only when} \quad k \in \{l+2,l,l-2\}\,.
    \]
    We compute
    \begin{eqnarray*}
        (\pi_{XY\L^{-1}}(\lambda))_{l,l} & = & (\pi_{XY}(\lambda))_{l,l} (\pi_{\L^{-1}}(\lambda))_{l,l}\\
        & = &\left(-i \frac{|\lambda|}{2}(2l+1) \right)(|\lambda|^{-1}(2l+1)^{-1})\\
        & = & -\frac{i}{2}\,,
    \end{eqnarray*}
     \begin{eqnarray*}
        (\pi_{XY\L^{-1}}(\lambda))_{l,l+2} & = & (\pi_{XY}(\lambda))_{l,l+2} (\pi_{\L^{-1}}(\lambda))_{l+2,l+2}\\
        & = &\left(i \frac{|\lambda|}{2}\sqrt{(l+1)(l+2)} \right)(|\lambda|^{-1}(2l+5)^{-1})\\
        & = & \frac{i}{2}\frac{\sqrt{(l+1)(l+2)}}{(2l+5)}\,,
         \end{eqnarray*}
         and
           \begin{eqnarray*}
        (\pi_{XY\L^{-1}}(\lambda))_{l,l-2} & = & (\pi_{XY}(\lambda))_{l,l-2} (\pi_{\L^{-1}}(\lambda))_{l-2,l-2}\\
        & = &\left(i \frac{|\lambda|}{2}\sqrt{l(l-1)} \right)(|\lambda|^{-1}(2l-3)^{-1})\\
        & = & \frac{i}{2}\frac{\sqrt{l(l-1)}}{(2l-3)}\,.
         \end{eqnarray*}
   The last one can be also written as
   $(\pi_{XY\L^{-1}}(\lambda))_{l+2,l}=
   \frac{i}{2}\frac{\sqrt{(l+2)(l+1)}}{(2l+1)}$.
   
   Finally since $\frac{\sqrt{(l+1)(l+2)}}{(2l+5)}, \frac{\sqrt{(l+2)(l+1)}}{(2l+1)}<1 $, we get that $\max_{k,l}|(\pi_{XY\L^{-1}}(\lambda))_{l,k}|=\frac{1}{2}$, which in turn implies that
   \[
    \|XY \L^{-1}\|_{\text{op}}=\frac{1}{2}\,.
   \]
   We note that the formulae for components of vector fields $X$ and $Y$ in \eqref{px} and \eqref{py} differ by multiplication by $i$, so we can carry out the same argument as above for products $X_i X_j \L^{-1}$ and $Y_i Y_j \L^{-1}$, as well as for $Y_j X_j \L^{-1}$, with the same estimate,
   and the proof is complete.
\end{proof}
For the setting of the the Heisenberg group $\mathbb{H}^n$ in Section \ref{SEC:intro}, we have $\G=\mathbb{H}^n$ and $N=2n$.
In particular, Proposition \ref{Heis} implies that, the constant $C_0$ as in \eqref{C0} in the case of the Heisenberg group $\mathbb{H}^n$, is equal to $C_0=n2^{1/2}$, which in turn implies that the condition $\sqrt{1-\varepsilon}C_0<1$ as in Theorem \ref{theorem2} is satisfied for $\varepsilon \in (1-(1/2n^2),1)$. Hence, Theorem \ref{theorem2} implies:
\begin{cor}\label{cor.H}
 Let $f \in L^2(\HN)$ and let the coefficients $(c_{ij}) \in L^{\infty}(\HN)$, $i,j=1,\ldots,2n$,   satisfy the condition \eqref{ell} and the Cordes condition \eqref{cordes} for some $\varepsilon  \in (1-(1/2n^2),1)$, with $N=2n$. Then  the equation 
  \[
    \sum_{i,j=1}^{2n}c_{ij}(x) Z_i Z_j u=f\,, \quad f \in L^2(\HN)\,,
    \]
 where $Z_i=X_i$, $Z_{n+i}=Y_i$, $i=1,\ldots, n$,
 has a unique solution $u \in {{\dot{\mathcal{H}}^{2}_{\mathcal{L}}(\HN)}}$. 
\end{cor}

\section{Acknowledgements}
Marianna Chatzakou and Michael Ruzhansky  are supported by the FWO Odysseus 1 grant G.0H94.18N: Analysis and Partial Differential Equations and by the Methu\-salem programme of the Ghent University Special Research Fund (BOF) (Grant number 01M01021). Marianna Chatzakou is a
postdoctoral fellow of the Research Foundation – Flanders (FWO) under the postdoctoral grant No 12B1223N. Michael Ruzhansky is also supported by EPSRC grant EP/R003025/2 and FWO Senior Research Grant G022821N.

\thanks{There is no conflict of interest. }


\begin{thebibliography}{BBGM12}


\bibitem[BD02]{BD02}
A. Buica and A. Domokos.
\newblock Nearness, accretivity and the solvability of nonlinear equations.
\newblock {\em Numer. Funct. Anal. Optim.}, 23(5,6):477-497, 2002.

\bibitem[Ca67]{C67}
S. Campanato.
\newblock Un risultato relativo ad equazioni ellitiche del secondo ordine di tipo non variazionale.
\newblock {\em Ann. Sc. Norm. Super. Pisa Cl. Sci.}, 21:701-707, 1967.

\bibitem[Ca89a]{C89a}
S. Campanato.
A Cordes type Condition for nonlinear non-variational Systems.
\newblock {\em Rend. Accad. Naz. Sci. Detta XL}, V Ser. 13, No. 1, (1989), 307-321.

\bibitem[Ca89b]{C89b}
S. Campanato.
\newblock A history of Cordes Condition for second order elliptic operators.
In: {\em ``Boundary Value Problems for Partial Differential Equations and Applications''}, J. L. Lions and C. Baiocchi (eds.), Research Notes in Applied Mathematics, Vol. 29, Masson, Paris, 1993, 319-325.



\bibitem[Ca94]{C94}
S. Campanato.
\newblock On the condition of nearness between operators.
\newblock {\em Ann. Mat. Pura. Appl.}, 167(4):243-256, 1994.

\bibitem[Co61]{C61}
H. O. Cordes.
\newblock Zero order a priori estimates for solutions of elliptic differential equations.
\newblock {\em Proc. Sympos. Pure Math.}, 4:157-166, 1961.

\bibitem[DM05]{DM05}
A. Domokos and J.J. Manfredi.
\newblock $C^{1,\alpha}$-relugalarity for $p$-harmonic functions in the Heisenberg group for $p$ near $2$. (English summary) The $p$-harmonic equation and recent advances in analysis. {\em Contemp. Math.}, 370:17-23. {\em American Mathemnatical Society}, Provience, RI, 2005.

\bibitem[DKY20]{DKY20}
D. Drivaliaris, Y. Karagiorgos and N. Yannakakis.
\newblock Nearness of nonlinear operators.
\newblock {\em Rend. Circ. Mat. Palermo (2)}, 70:1051-1060, 2021


\bibitem[FR16]{FR16}
V. Fischer and M. Ruzhansky. {\em Quantization on nilpotent Lie groups}. Progress in Mathematics, Vol. 314, Birkh\"{a}user, 2016. 

\bibitem[FRT21]{FRT21}
V. Fischer, M. Ruzhansky and C. Taranto. Subelliptic Gevrey spaces,
\newblock \textit{Math. Nachr.}, 294:265-285, 2021.

\bibitem[Fo73]{F73}
G. B. Folland.
\newblock A fundamental solution for a subelliptic operator.
\newblock{\em Bull. Amer. Math. Soc.}, 79(2): 373-376, 1973.

\bibitem[FS82]{FS82}
G. Folland and E. M. Stein.
\newblock  {\em Hardy spaces on homogeneous groups}.
\newblock Vol. 28 of Math. Notes, Princeton University Press, Princeton, N.J., 1982.




\bibitem[GT01]{GT01}
G. Gilbarg and N. S. Trudinger.
\newblock {\em Elliptic Partial Differential Equations of Second Order.}
\newblock Springer, Berlin, 2001.


\bibitem[Ho67]{H67}
L. H\"ormander.
\newblock  Hypoelliptic second order differential equations.
\newblock {\em Acta Math.} 119 (1967), 147-171.

\bibitem[JL88]{JL88}
D. Jerison and J. M. Lee. Extremals for the Sobolev inequality on the Heisenberg group and the CR Yamabe problem. {\it J. Amer. Math. Soc.},  1(1):1-13, 1988.

\bibitem[Ka17]{K17}
N. Katzourakis.
\newblock Generalilsed solutions for fully nonlinear PDE systems and existence-uniqueness theorems.
\newblock{\em Differ. Equ.}, 263:641-686, 2017.

\bibitem[MPS00]{MPS00}
A. Maugeri, D. K. Palagachev, and L. G. Softova. \newblock {\em Elliptic and Parabolic Equations with Discontinuous Coefficients.}
\newblock Wiley-VCH Verlag Berlin GmbH, Berlin, 2000

\bibitem[Pu66]{P66}
C. Pucci. 
\newblock Limitazioni per soluzioni di equazioni ellittiche.
\newblock {\em Ann. Mat. Pura Appl.}, IV Ser. 74:15-30, 1966.


\bibitem[Ta65]{T65}
G. Talenti.
\newblock Sopra una classe di equazioni ellittiche a coefficienti discontinui.
\newblock {\em Ann. Mat. Pura Appl.}, IV Ser. 69, 285-304, 1965.

\bibitem[Ta04]{T04}
A. Tarsia.
\newblock On Cordes condition and Campanato conditions.
\newblock {\em Arch. Inequal. Appl.}, 2(1):25-39, 2004.



\end{thebibliography}
\end{document}